\documentclass[10pt,a4paper]{article} 

\usepackage{anysize}
\usepackage[format = hang]{caption}
\usepackage{amsmath,amsthm,verbatim,amssymb,amsfonts,amscd, graphicx}
\usepackage{graphics,natbib}
\usepackage{titlesec,footmisc}
\usepackage{hyperref} 
\usepackage{color}
\usepackage{listings}
\usepackage{booktabs} 

\theoremstyle{plain}

\newtheorem{lemma}{Lemma}

\theoremstyle{definition}

\def\E{{\rm E}}

\def\midd{\,|\,}
\def\pr{{\rm pr}}
\def\dd{{\rm d}}
\def\hatt{\widehat}
\def\arr{\rightarrow}
\def\N{{\rm N}}
\def\half{\hbox{$1\over 2$}}
\def\eps{\varepsilon}
\def\Gam{{\rm Gamma}}
\def\binom{{\rm Bin}}

\def\beq{\begin{eqnarray}}
\def\eeq{\end{eqnarray}}

\def\beqn{\begin{eqnarray*}}  
\def\eeqn{\end{eqnarray*}}

\def\E{{\rm E}}

\def\dd{{\rm d}}
\def\N{{\rm N}}

\def\Pr{P}
\def\pr{{\rm pr}}

\def\quadandquad{\quad {\rm and} \quad}
\def\arr{\rightarrow}
\def\hatt{\widehat}

\def\sumin{\sum_{i=1}^n}

\def\eps{\varepsilon}
\def\half{\hbox{$1\over2$}}

\def\quart{\hbox{$1\over4$}}

\def\rootn{\sqrt{n}}

\def\midd{\,|\,}

\def\pois{{\rm Pois}}
\def\ubiq{\sqrt{2\pi}}

\titleformat{\section}{\normalfont\large\sc\centering}{\thesection}{1em}{}
\titleformat{\subsection}[runin]{\normalfont\large\bfseries}{\thesubsection}{1em}{}
\numberwithin{equation}{section} 
\renewenvironment{abstract}
               {\list{}{\rightmargin\leftmargin}%
                \item[\text{\hspace{8.5mm}\sc Abstract.}]\relax}
               {\endlist}

\begin{document}

\def\idag{October 2024}
\def\heute{\idag}
\begingroup
\begin{centering} 

\Large{\bf Probability Proofs for Stirling (and More): \\
     the Ubiquitous Role of $\mathbf {\sqrt{2\pi}} $ }\\[0.8em]
\large{\bf Nils Lid Hjort$^1$ and Emil Aas Stoltenberg$^2$ } \\[0.3em] 
\small {\sc $^1$Department of Mathematics, University of Oslo} \\[0.3em]
\small {\sc $^2$BI Norwegian Business School, Oslo } \\[0.3em]
\small {\sc {\heute}}\par
\end{centering}
\endgroup

\begin{abstract}
\noindent
The Stirling approximation formula for $n!$ dates from 1730.
Here we give new and instructive
proofs of this and related approximation formulae
via tools of probability and statistics.
There are connections to the Central Limit Theorem
and also to approximations of marginal distributions
in Bayesian setups. Certain formulae emerge
by working through particular instances, 
some independently verifiable but others perhaps not. 
A particular case yielding new formulae is that
of summing independent uniforms, related to the
Irwin--Hall distribution. Yet further proofs of the Stirling
flow from examining aspects of limiting normality
of the sample median of uniforms, and from these
again we find a proof for the Wallis product formula for $\pi$.

\end{abstract}

\section{Introduction} 
\label{section:intro}

The Stirling approximation formula is a famous one, stating that
\beq
n!\doteq n^n\exp(-n)\sqrt{2\pi\, n},
\quad {\rm in\ the\ sense\ of} \quad
{n!\over n^{n+1/2}\exp(-n)}\arr \ubiq.
\label{eq:stirling}
\eeq
Intriguingly and surprisingly, starting just by multiplying
$1\!\cdot\!2$,
$1\!\cdot2\!\cdot3$,
$1\!\cdot2\!\cdot3\!\cdot4,\ldots$, 
the formula turns out to involve the eternal mathematical
constants $e = 2.718282\ldots$ and $\pi = 3.141593\ldots$. 
Our aim here is to tie the Stirling and also related formulae
to the Central Limit Theorem (CLT), with the limiting normality
being the leading clue to both $e$ and the for statisticians
famous quantity $\ubiq$. There are also other Stirling 
connections to probability and statistics, including
approximations for marginal distributions in Bayesian setups. 

There are of course many different proofs in the literature,
going back all the way to \citet{Stirling1730} and \citet{Demoivre1730},
with ensuing scholarly articles to understand
precisely how they arrived at their findings,
what the differences were, the degree to which Stirling
deserves the name without de Moivre, and so on.
We choose to give a brief review of these historical themes
in our penultimate separate Section \ref{section:history},
with further comments concerning associated issues touched
on in our article, such as the connection from Wallis' 1656
product formula for $\pi$ to Stirling 1730,
the use of Laplace 1774 approximations for integrals, 
the Irwin--Hall distribution from 1927, among other topics
from the history of $1\!\cdot2\!\cdot3\!\cdot4,\ldots$. 
Readers mainly interested in our various CLT and
marginal distribution connections may then read on,
through the main sections, without necessarily
caring about the historical footnotes. We do point to one
such here, however, namely \citet{Pearson24},
since it directly pertains to the core questions
discussed in our article. He noted, 
``I consider that the fact that Stirling showed that
de Moivre's arithmetical constant was $\ubiq$
does not entitle him to claim the theorem''
(with various later scholars disagreeing with him
in this regard; see indeed Section \ref{section:history}). 
The issue is both of a technical nature,
``where does $\ubiq$ come from'',
and a key theme regarding sorting out who found what, how, and when.

With the benefit of some extra $295$ years of probability, we as statisticians of today might be slightly less surprised
than was de Moivre, in 1729, having received a letter
from Stirling; we're used to seeing $\ubiq$
as part of our normal workload.
There is indeed a certain literature relating
Stirling to probability themes.
\citet[Ch.~5]{Feller68} reaches the Stirling formula
in his classic book, via careful study of its logarithm,
supplemented with limiting normality calculus for the binomial
(and appears to miss that one of his exercises
gives the Stirling in a simpler fashion;
check with Section~\ref{section:history} again). 
\citet{Hu88} and \citet{Walsh95} have contributed
arguments and proofs related to approximations
from the CLT, specifically with the Poisson distribution.
Some of these ``easy proofs'' in the literature
have needed some further finishing polish to be fully accurate,
however, as commented upon in \citet{BlythPathak86}.
These authors also contribute a new proof, based on
inversion of characteristic functions. 
\citet{DiaconisFreedman86} give a neat
proof of Stirling, via Laplace approximations of
the gamma function, at the outset without a connection
to probability, though they arrived at their proof
via de Finetti theorems. 


In Section \ref{section:CLTcases} we explain the main idea
tying the CLT to Stirling, with emphasis on the
special cases of the Poisson, the gamma, the binomial distributions.
Then in Section \ref{section:uniforms} we establish
a further connection, with partly new formulae,
via the Irwin--Hall distribution, for the sum
of independent uniforms. Yet another connection,
via the density of the uniform median, is worked
with in Section \ref{section:viamedians}. This is also
seen to yield a proof for the 1656 Wallis product formula
for $\pi$. We go on to see how Stirling emerges in yet
further ways, via Laplace and approximations
to certain marginal distributions, in Section \ref{section:Laplace}.
As mentioned we then have a separate
Section \ref{section:history} with historical notes,
before we offer a few complementary remarks in
our final Section \ref{section:concludingremarks}. 


\section{The CLT connection}
\label{section:CLTcases} 


To set the stage for what shall give us new proofs and insights,
also for other related formulae,
consider in general terms i.i.d.~variables $X_i$ with mean $\xi$
and standard deviation $\sigma$,
and with partial sums $Y_n=\sumin X_i$.
With `$\arr_d$' denoting convergence in distribution, 
$Z_n=(Y_n-n\xi)/(\rootn\sigma)\arr_d Z$, a standard normal,
and 
\beq
L_n(c)=\E\,Z_n\,I(0\le Z_n\le c)
\arr L(c)=\E\,Z\,I(0\le Z\le c),
\label{eq:LntoL} 
\eeq 
with this limit equal to
\beq
L(c)=\int_0^c z\phi(z)\,\dd z
=\phi(0)-\phi(c)=(2\pi)^{-1/2}\{1-\exp(-\half c^2)\}. 
\label{eq:hereisLc} 
\eeq
In particular, with $c=\infty$, the mean of the zero-truncated
$Z_n$ tends to $1/\ubiq$. The key ingredient securing
moment convergence of $L_n(c)$ to $L(c)$, automatically
valid for each CLT application, is uniform integrability,
via $\E\,Z_n^2=1$ and $\E\,Z^2=1$;
see \citet[Ch.~2]{HjortStoltenberg25}. 

In cases where we manage to have an interesting formula
for $L_n(c)$, therefore, we have learned something.
Here we work through special cases for the CLT,
aiming for situations with clear and insightful
$L_n(c)\arr L(c)$ formulae. 

\subsection{Poisson.} 

Consider first the Poisson case, where $Z_n=(Y_n-n)/\rootn\arr_d Z$,
for $Y_n\sim\pois(n)$. With $p_n(j) =\exp(-n)n^j/j!,\,j = 0,1,2,\ldots$ the Poisson
probabilities, one finds that 
\beqn
\sum_{j\ge n}(j-n)p_n(j)
&=&np_n(n)-np_n(n)
+(n+1)p_n(n+1)-np_n(n+1) \\
& &\qquad +\, (n+2)p_n(n+2)-np_n(n+2) +  \cdots 
\eeqn 
is a telescoping series, in sum equal to $np_n(n)$. This leads to
\beqn
L_n(\infty)=\sum_{j\ge n}{j-n\over \rootn}p_n(j)
   =\rootn\exp(-n){n^n\over n!}, 
\eeqn 
which by (\ref{eq:LntoL}) tends to $1/\ubiq$, proving the Stirling. 

Also, for a finite $c$ in (\ref{eq:LntoL}), the telescoping nature
of the sum leads to
\beqn
L_n(c)=\sum_{n\le j\le n+c\rootn} {j-n\over \rootn}p_n(j)
&=&\rootn \{p_n(j)-p_n(n+c\rootn)\} \\
&=&\rootn \exp(-n){n^n\over n!}
   \Bigl\{1 - {n^{[c\rootn]}\over (n+1)\cdots([n+c\rootn])}\Bigr\} 
\eeqn 
with $[x] = \max\{z \in \mathbb{Z} \colon z \leq x\}$
the integer value of $x$. From \eqref{eq:hereisLc}, therefore,
we infer that
\beqn
\lim_{n\arr\infty} {n^{[c\rootn]}\over (n+1)\cdots([n+c\rootn])} = \exp(-\half c^2). 
\eeqn 
This may be verified independently, via logarithms and inspection of
$\sum_{1\le j\le c\rootn}\log(1+i/n)$; the present point is that we
get the formula from probability theory and the CLT. 

\subsection{Gammas.}
\label{subsection:gammas}

Let next $Y_n\sim\Gam(n,1)$, where again $Z_n=(Y_n-n)/\rootn\arr_d Z$.
Write $g_n(x)=\Gamma(n)^{-1}x^{n-1}\exp(-x)$ for the density,
with $\Gamma(z) = \int_0^{\infty} t^{z- 1}\exp(-t)\,\dd t $
the gamma function, and let
$G_n(x)=1-\exp(-x)\{1+x+\cdots+x^{n-1}/(n-1)!\}$ be the cumulative. 
One sees that $xg_n(x)=ng_{n+1}(x)$, which leads to
\beqn
L_n(c)=\int_n^{n+c\rootn} {x-n\over \rootn}g_n(x)\,\dd x
   =\rootn\int_n^{n+c\rootn} \{g_{n+1}(x)-g_n(x)\}\,\dd x. 
\eeqn 
First, for $c=\infty$, one finds from~\eqref{eq:LntoL}
and \eqref{eq:hereisLc}, i.e., the CLT plus the mean to mean argument, that 
\beqn
L_n(\infty)=\rootn \exp(-n)n^n/n!\arr \ubiq, 
\eeqn
which is precisely the Stirling formula \eqref{eq:stirling}, again. 
Working out the case for finite $c$, one finds the additional formula 
\beq
\lim_{n\arr\infty} \exp(-c\rootn)(1+c/\rootn)^n=\exp(-\half c^2).
\label{eq:newguy}
\eeq 
This may again be verified independently, but here we get it for free
from the CLT and the mean to mean argument \eqref{eq:LntoL}. 

For the gamma case we may also work directly with the density of $Z_n$,
which becomes
\beqn
h_n(z)=\rootn g_n(n+\rootn z)
&=&\rootn\,\Gamma(n)^{-1}(n+\rootn z)^{n-1}\exp(-n-\rootn z) \\
&=&{\rootn \over n!} n^n\exp(-n)(1+z/\rootn)^{n-1}\exp(-\rootn z). 
\eeqn 
Appealing to the Scheff\'e lemma, for convergence of densities, 
we see that we have proven the Stirling in one more way,
via the (\ref{eq:newguy}) formula just found. 

The chi-squared distribution is merely a scale factor away
from the gamma, so going through details for $Z_n=(\chi^2_n-n)/(2n)^{1/2}$
will essentially bring out the same: Stirling holds, via the CLT
for these distributions. 

\subsection{Binomial.} 

The symmetric binomial distribution also lends itself
nicely to working with the truncation and its mean.
With $X_1,\ldots,X_n$ i.i.d.~Bernoulli, we have $Y_n=\sumin X_i$
a binomial $(n,\half)$, and of course
$Z_n = (Y_n-n\half)/(\half\rootn)\arr_d Z$.
Write $b_n(j)={n\choose j}(\half)^n,\, j = 0,1,\ldots,n$
for the binomial $(n,\half)$ probabilities. We need to work with
$\sum_{j\ge n/2} (j-n/2)b_n(j)$. Note that
\beqn
j{n\choose j}=n{n-1\choose j-1} \quadandquad
b_n(j)=\half b_{n-1}(j-1) + \half b_{n-1}(j),
\eeqn 
which leads to 
\beqn
(j-n/2)b_n(j)
   =\half n\{b_{n-1}(j-1)-\half b_n(j)\}
   =\quart n \{b_{n-1}(j-1)-b_{n-1}(j)\}. 
   \eeqn
The telescoping nature of these terms then gives us
\beqn
\sum_{j\ge n/2} (j-n/2)b_n(j)=\quart n \{b_{n-1}(\half n-1)-(\half)^{n-1}\},
\eeqn 
taking for simplicity $n$ even. Hence, before checking the details of the $b_{n-1}(j)$, we know from the CLT and the mean to mean argument that
\beqn
L_n(\infty)=\sum_{j\ge n/2}{j-n/2\over \half\rootn} b_n(j)
   =\half \rootn \{b_{n-1}(\half n-1)-(\half)^{n-1}\} \arr 1/\ubiq. 
\eeqn
This translates to
\beqn
\half\rootn {n-1\choose n/2-1}(\half)^{n-1}\arr 1/\ubiq,
\quad \text{or}\quad
\half(n+1)^{1/2} {n\choose (n-1)/2}(\half)^{n}\arr 1/\ubiq,
\eeqn
the latter valid for $n$ odd.
But this is essentially the Stirling statement~\eqref{eq:stirling},
again; in this regard, see the details of Section \ref{section:viamedians}. 
For another relevant footnote, if one starts with
$n!\doteq a\exp(-n)n^{n+1/2}$, as de Moivre did, around 1730,
without yet knowing the constant $a$, one learns from the
CLT implied details above that one indeed must have $a=\ubiq$. 

We learn more from the general setup above,
with $L_n(c)$ of (\ref{eq:LntoL}). With $j_n=[\half n+\half c\rootn]$,
and the telescoping nature of the terms in the sum above, 
\beqn
\sum_{n/2\le j\le j_n} {j-n/2 \over \half\rootn}b_n(j)
   &=&\half\rootn \{b_{n-1}(\half n-1)-b_{n-1}(j_n)\} \\
   &=&\half\rootn b_{n-1}(\half n-1)
      \Bigl\{1-{b_{n-1}(j_n)\over b_{n-1}(\half n-1)}\Bigr\}.
\eeqn 
We are hence also learning, in the process, that 
\beqn
\lim_{n\arr\infty} {b_n([\half n+\half c\rootn])\over b_n(\half n)}
=\lim_{n\arr\infty} {n\choose [\half n+\half c\rootn]}
   \Big/{n\choose \half n} = \exp(-\half c^2). 
\eeqn 




\section{Sums of uniforms and the Irwin--Hall distribution} 
\label{section:uniforms} 
Consider the sum $Y_n=\sumin X_i$ of i.i.d.~uniforms on $(0,1)$, these
having mean $\xi=\half$ and variance $\sigma^2=1/12$.
The probability density $f_n(y)$ of $Y_n$ is a somewhat awkward one,
with different $(n-1)$-order polynomials over the intervals $[j-1,j]$
knotted together, for $j=1,\ldots,n$, 
worked out in two independent articles in the same issue of Biometrika in 1927.
What we may state, before coming to, or for that matter
independently of the details of $f_n$ and its moments, is that
with $I_n=\int_0^{n/2} yf_n(y)\,\dd y$ and
$J_n=\int_{n/2}^{n} yf_n(y)\,\dd y$, we do have 
\beq
(I_n-\half n\xi)/(\sigma\rootn)\arr-1/\ubiq,
\quad \text{and}\quad 
(J_n-\half n\xi)/(\sigma\rootn)\arr 1/\ubiq. 
\label{eq:weknowthis}
\eeq
To see how this follows from~\eqref{eq:LntoL} and \eqref{eq:hereisLc},
note that $Z_n = (Y_n - n\xi)/(\sigma \rootn) \arr_d \N(0,1)$,
and $I_n + J_n = n \xi$, so  
\beqn
0
&=& (I_n - \half n\xi) + (J_n - \half n\xi) \\
&=& \E\, (Y_n - n\xi)I\{Y_n \leq \half n\}
+ \E\, (Y_n - n\xi)I\{\half n \leq Y_n \leq n\} \\
&=& \rootn\sigma\, \E\, Z_n I\{- \half \rootn/\sigma \leq Z_n \leq 0\}
+ \rootn\sigma \,\E\, Z_n I\{0 \leq Z_n \leq \half \rootn/\sigma\}. 
\eeqn 
The only thing differing between the limits of
$(I_n - \half n\xi)/(\rootn \sigma)$ and
$(J_n - \half n\xi)/(\rootn \sigma)$, therefore, 
are their signs, but $|Z_n| \leq \sqrt{n}/(2\sigma)$,
so by~\eqref{eq:hereisLc},
\beqn
\E\, Z_n I\{0 \leq Z_n \leq \half \rootn/\sigma\} = L_n(\infty) \to 1/\ubiq. 
\eeqn 
The perhaps gently misnomed Irwin--Hall distribution
is triangular for $n=2$, then a mix of 2nd order polymomials
for $n=3$, and becomes steadily smoother and of course more and
more normal looking for $n=4,5,\ldots$, via piecewise polynomials.
Formulae for the density and cumulative of the Irwin--Hall
distribution on $[0,n]$ can be found in the literature,
e.g.~in these forms, 
\beqn
f_n(y)
={1\over (n-1)!}\sum_{j=0}^{ [y] } (-1)^j{n\choose j}(y-j)^{n-1}
\quadandquad  
F_n(y)
={1\over n!}\sum_{j=0}^{ [y] } (-1)^j {n\choose j}(y-j)^n. 
\eeqn
Proofs given are via somewhat cumbersome induction,
or inversion of generating functions, 
but see \citet*{Marengo17} for a different geometric type of
argument. It is nice now to have an occasion and a reason
for finding a new formula for a key quantity,
namely the truncated moment, for a more than hundred years
old distribution. 

\begin{lemma}
\label{lemma:hereisIn}
{{\rm   
For the Irwin--Hall density $f_n$, with $n$ even, we have 
\beqn
I_n
=\int_0^{n/2} yf_n(y)\,\dd y
=\quart n-{n^{n+1}\over (n+1)!}(\half)^{n+1}
   \sum_{j=0}^{n/2} (-1)^j{n\choose j}(1-2j/n)^{n+1}.
\eeqn }} 
\end{lemma} 

\begin{proof}
We note first, via partial integration, that
\beqn
\int_0^a(1-F_n)\,\dd y = a\{1-F_n(a)\} + \int_0^a yf_n(y)\,\dd y
\quad {\rm for\ any\ }a.
\eeqn
With $s_n=\int_0^{n/2} F_n\,\dd y$, therefore, we have
$I_n=\int_0^{n/2}(1-F_n)\,\dd y-\quart n=\quart n-s_n$. 
For the latter we have
\beqn
s_n
&=&\int_0^{n/2} {1\over n!}\sum_{j=0}^{ [y] }
(-1)^j{n\choose j}(y-j)^n\,\dd y \\
&=&{1\over n!}\sum_{k=0}^{n/2-1} \int_k^{k+1} \sum_{j=0}^{ [y] }
(-1)^j{n\choose j}(y-j)^n\,\dd y \\
&=&{1\over n!}\sum_{k=0}^{n/2-1} \sum_{j=0}^{k}
(-1)^j{n\choose j}\int_k^{k+1}(y-j)^n\,\dd y \\
&=&{1\over n!}\sum_{k=0}^{n/2-1} \sum_{j=0}^{k}
(-1)^j{n\choose j}{1\over n+1}\{(k+1-j)^{n+1}-(k-j)^{n+1}\} \\
&=&{1\over (n+1)!}\sum_{j=0}^{n/2-1} (-1)^j{n\choose j}
   \sum_{k=j}^{n/2-1} \{(k+1-j)^{n+1}-(k-j)^{n+1}\}. 
\eeqn
This may also be expressed as
\beqn
s_n&=&{1\over (n+1)!} 
\Bigl[ \{1^{n+1}-0^{n+1}+\cdots+(n/2)^{n+1}-(n/2-1)^{n+1}\} \\
& &\quad -{n\choose 1} \{1^{n+1}-0^{n+1}+\cdots+(n/2-1)^{n+1}-(n/2-2)^{n+1}\} \\
& &\quad +{n\choose 2} \{1^{n+1}-0^{n+1}+\cdots+(n/2-2)^{n+1}-(n/2-3)^{n+1}\} 
  +\cdots\Bigr] \\
&=&{1\over (n+1)!} [(n/2)^{n+1} - {n\choose 1}(n/2-1)^{n+1}
  +{n\choose 2}(n/2-2)^{n+1} - \cdots] \\
&=&{1\over (n+1)!} \sum_{j=0}^{n/2} (-1)^j{n\choose j}(n/2-j)^{n+1}, 
\eeqn 
proving the lemma.
\end{proof}


Via these unchartered pathways in the terrain of truncated
moments for Irwin--Hall we have reached explicit 
expressions for $s_n$ and $I_n=\quart n-s_n$.
The statements (\ref{eq:weknowthis}) are seen to be equivalent to 
\beqn
{s_n\over \rootn}
=\half {n^{n+1/2} \over (n+1)!}
\sum_{j=0}^{n/2} (-1)^j{n\choose j}(\half)^n(1-2j/n)^{n+1}
\arr \sigma/\ubiq. 
\eeqn 
From Stirling we have
\beqn
{n^{n+1/2}\over (n+1)!}
\doteq { n^{n+1/2}\over (n+1)^{n+3/2}\exp(-(n+1)) \ubiq }
={\exp(n)\over n+1} \Bigl({n\over n+1}\Bigr)^{n+1/2}{e\over \ubiq}, 
\eeqn
in the sense that the ratio between the two sides converges to 1.
This invites examining the product form
$s_n/\rootn=a_nb_n$, with 
\beqn
a_n={n+1\over \exp(n)}{n^{n+1/2}\over (n+1)!},
\quad 
b_n=\half{\exp(n)\over n+1}
   \sum_{j=0}^{n/2}(-1)^j{n\choose j}(\half)^n(1-2j/n)^{n+1}. 
\eeqn
Here $a_n\arr 1/\ubiq$, and by mathematical necessity 
the complicated looking $b_n$ needs to tend to $\sigma$.
So, in a rather roundabout fashion, via the CLT
and moment convergence, we have managed to prove not
merely the Stirling formula, once again, but also the
impressive looking
\beqn
\lim_{n\arr\infty} {1\over n}\exp(n)
\sum_{j=0}^{n/2}(-1)^j{n\choose j}(\half)^n(1-2j/n)^{n+1}
   =1/\sqrt{3}. 
\eeqn 


\section{Stirling and Wallis from the median of a uniform sample}
\label{section:viamedians} 

Here we shall prove the Stirling approximation formula
(\ref{eq:stirling}) once more, and as a by-product
also the famous Wallis 1656 product formula for $\pi$, 
starting again with a simple i.i.d.~uniform sample
$U_1,\ldots,U_n$.

\subsection{Stirling, again.}
\label{sec:stirlingagain}

In Section \ref{section:uniforms} we worked with the sum
of the uniforms, but now our attention is on their median,
say $M_n$. With $n=2m+1$ odd, for simplicity,
an easy and well-known argument gives its density as
\beqn
g_n(x)={(2m+1)!\over m!\,m!}x^m(1-x)^m \quad {\rm for\ }x\in[0,1], 
\eeqn 
which is a Beta $(m+1,m+1)$. Via a well-known representation
for the Beta, in terms of a ratio of Gammas, we may write
$M_n=U_{m+1}/(U_{m+1}+V_{m+1})$,
where $U_{m+1}=A_1+\cdots+A_{m+1}$ and $V_{m+1}=B_1+\cdots+B_{m+1}$,
say, involving i.i.d.~unit exponentials. It then follows
straightforwardly and instructively that
\beqn
(m+1)\{U_{m+1}/(m+1)-1\}\arr_d Z_1,\quad 
(m+1)\{V_{m+1}/(m+1)-1\}\arr_d Z_2, 
\eeqn
with these limits being independent standard normals,
and that $Z_n=2\rootn(M_n-\half)\arr_d Z=(Z_1-Z_2)/\sqrt{2}$,
a standard normal,
via an application of the delta method. The two limiting 
standard normals $Z_1$ and $Z_2$ featured here
are incidentally just as in Section \ref{subsection:gammas},
but here they play a different role,
just to establish that $Z_n\arr_d Z$. 

We learn more by working with the density of $Z_n$, which can be written 
\beqn
{1\over 2\rootn} g_n(\half+{z\over 2\rootn}) 
&=&{1\over 2\rootn}{(2m+1)!\over m!\,m!}
\{(\half+{z\over 2\rootn})(\half-{z\over 2\rootn})\}^m \\
&=& \half \rootn {(2m)!\over m!\,m!}
   \Bigl(1 - {z^2\over n}\Bigr)^m(\half)^{2m}.  
\eeqn 
Since $(1-z^2/n)^m\arr\exp(-\half z^2)$, and we already know
that $Z_n\arr_d Z$, we must by necessity have
\beq
c_n=
\half(2n+1)^{1/2} {2n\choose n}(\half)^{2n}\arr c=1/\ubiq. 
\label{eq:cntoc}
\eeq
This is actually related to what \citet{Demoivre1730} was
working on (without finding the $\ubiq$ connection),
good approximations to the middle term of
the binomial expansion of $(1+1)^{2n}$.
With $Y_{2n}$ a binomial $(2n,\half)$, we may translate
the above to the middle binomial probability
\beq
\Pr(Y_{2n}=n)=b_{2n}(n,\half)
   ={2n\choose n}(\half)^{2n}
   \doteq {1\over \sqrt{\pi\,n}}, 
\label{eq:middlebinomial}
\eeq
associated also with the recurrent nature of symmetric
random walks; see Remark C in Section \ref{section:concludingremarks}. 

By inserting for $n!$, Stirling{'}s formula implies the
limit in~\eqref{eq:cntoc}, i.e., that $c_n \arr 1/\ubiq$;
but the implication also goes the other way around.
To see this, recall that by the trapezoidal approximation for integrals, 
\beq
\log n! - \half \log n = \int_0^n \log x\,\dd x + O(1/n)
= n \log n - n + 1  + O(1/n).
\label{eq:trapezoidal}
\eeq 
as $n$ tends to infinity. Since $(2n + 1)/2n \to 1$,
clearly $c_n \doteq \half(2n)^{1/2} \{(2n)!/(n!)^2\}(\half)^{2n}$,
which combined with~\eqref{eq:cntoc} entails that 
\beqn
\frac{\sqrt{2\pi n} \{(2n)!/n!\}2^{-2n - 1/2}}{n!} \arr 1. 
\eeqn  
If we take the logarithm of $\{(2n)!/n!\}2^{-2n - 1/2}$
and use the trapezoidal rule, 
\beqn
\log\big( \{(2n)!/n!\}2^{-2n - 1/2} \big)
&=&
\log ( (2n)!)  - \log (n!) - 2n \log 2 - \half \log 2 \\
& =&
\{\log ( (2n)!) - \half \log (2n)\}
    - \{ \log (n!) - \half \log(n) \}- 2n \log 2 \\
& =& \{2n \log(2n) - 2n + 1\} - \{ n \log n - n + 1\} - 2n \log 2 + O(1/n) \\
& =& n \log n - n + O(1/n),
\eeqn 
which shows that $\{(2n)!/n!\}2^{-2n - 1/2} \doteq n^n \exp(-n)$, once again proving the Stirling{'}s formula. This proof is, admittedly, very close to the standard proof of Stirling, the point is, yet again, that we get $\ubiq$ for free from the CLT. 

\subsection{The Wallis product.}

An intriguing formula, from \citet{Wallis1656}, says that
\beq
{\pi\over 2}
={2\over 1}\cdot{2\over 3}\cdot
 {4\over 3}\cdot{4\over 5}\cdot
 {6\over 5}\cdot{6\over 7}\cdot
 {8\over 7}\cdot{8\over 9}\cdots
=\prod_{j=1}^\infty {2j\over 2j-1}{2j\over 2j+1}. 
\label{eq:wallis} 
\eeq
Here we derive this via our Stirling efforts.
For the product up to $n$, write
\beqn
w_n=\prod_{j=1}^n {2j\over 2j-1}{2j\over 2j+1}
={1\over 2n+1} \prod_{j=1}^n { (2j)^4\over \{(2j)(2j-1)\}^2}
={1\over 2n+1} 2^{4n} \Big/ {2n\choose n}^2. 
\eeqn 
We recognise its square root, from (\ref{eq:middlebinomial}), 
and have 
\beqn
1/\sqrt{w_n}=(2n+1)^{1/2} {2n\choose n}(\half)^{2n} \arr (2/\pi)^{1/2},
\eeqn
proving (\ref{eq:wallis}). 
Enthrallingly, we have explained via understanding
the behaviour of the median in uniform samples
that the Wallis formula must be true. 

\section{Stirling via Laplace approximations} 
\label{section:Laplace} 

Suppose $g(x)$ is a smooth function with a clear maximum
$g_{\max}=g(x_0)$ at position $x_0$, and write $c=-g''(x_0)$.
Then the Laplace approximation is 
\beq
\int \exp(g)\,\dd x\doteq
\int\exp\{g(x_0)-\half c(x-x_0)^2\}\,\dd x
=\exp(g_{\max})\, \sqrt{2\pi/c}. 
\label{eq:laplace} 
\eeq
The $\doteq$ is to be read as `approximately equal to' in the sense that the ratio between the left and right hand sides tends to $1$ provided the steepness of the maximum, i.e.~the size of $c$,
increases. Note that the $\ubiq$ factor comes
from the familiar normal integral, not stemming from
probability, per se, but from the natural Taylor expansion
approximation in the exponent of the function being integrated. 

Now try this with the function $g(x)=n\log x-x$,
for $x$ positive. It has $g'(x)=n/x-1$ and $g''(x)=-n/x^2$; 
hence $x_0=n$, $g_{\max}=n\log n-n$, and $c=1/n$. We find
\beqn
\int_0^\infty x^n\exp(-x)\,\dd x=n!\doteq n^n\exp(-n)\sqrt{2\pi\,n}, 
\eeqn 
i.e.~Stirling, once again. This is essentially what
\citet{DiaconisFreedman86} do, but presented somewhat
differently. The argument may be made rigorous by letting
$n$ increase and assess the error made in the
Laplace approximation (\ref{eq:laplace}). 

The Laplace argument is known in model selection statistics
for being an ingredient in the BIC, the Bayesian Information
Criterion, see \citet[Ch.~3]{ClaeskensHjort08}. One version
of the essential computation there is as follows.
Consider data from a one-dimensional parametric model $f(x,\theta)$,
of sample size $n$, leading to the likelihood function
$L_n(\theta)$, say, with ensuing log-likelihood function $\ell_n(\theta)$.
Letting $\hatt\theta$ be the maximum likelihood estimator,
with $L_{n,\max}=L_n(\hatt\theta)$ and observed Fisher information
$J_n=-\ell_n''(\hatt\theta)$, we may integrate over
the parameter region to find 
\beq
\int L_n(\theta)\,\dd\theta
\doteq L_{n,\max} (2\pi/J_n)^{1/2}. 
\label{eq:bic} 
\eeq 
The left hand side is recognised as being the
marginal distribution for the data, in the Bayesian setup
with a flat prior for the parameter. 

It is now instructive to see how this approximation pans
out for a few models. Consider first the case of
a single $X\sim\pois(\theta)$, with likelihood
$\exp(-\theta)\theta^x / x!$ maximised for $\hatt\theta=x$.
The integral in (\ref{eq:bic}) is simply 1,
and the $L_{n,\max}$ is $\exp(-x)x^x/x!$.
The BIC approximation by almost magic delivers Stirling: 
\beqn
1\doteq \{\exp(-x)x^x/x!\}\,(2\pi x)^{1/2}.
\eeqn 
Examination of the error involved in the basic approximation
(\ref{eq:laplace}) shows how this simple Poisson argument
may be made rigorous, by letting $x$ increase.
We may also work through the case of $X_1,\ldots,X_n$
i.i.d.~from the Poisson. Writing $Z=\sumin X_i$, the
likelihood is proportional to $\exp(-n\theta)\theta^z$.
The maximum likelihood estimator is $\hatt\theta=z/n$,
the sample average. The BIC approximation (\ref{eq:bic})
is seen to become
\beqn
    {z!\over n^{z+1}}\doteq \exp(-z) \Bigl({z\over n}\Bigr)^z
       \ubiq {z^{1/2}\over n}, 
\eeqn 
i.e.~Stirling, again, now expressed in terms of $z$. 

Next consider $X_1,\ldots,X_n$ i.i.d.~from the exponential model
$\theta\exp(-\theta x)$ for $x$ positive.
The likelihood is $\theta^n\exp(-z\theta)$, with $z=\sumin x_i$,
leading to the maximum likelihood estimator
$\hatt\theta=n/z$, the inverse of the sample average.
The BIC approximation formula is seen to lead to
Stirling again, lo {\it\&} behold: 
\beqn
{n!\over z^{n+1}}\doteq
      \Bigl({n\over z}\Bigr)^n \exp(-n)\ubiq {\rootn\over z}. 
\eeqn 

We allow ourselves including also the binomial model here.
With $X$ a binomial $(n,p)$, and with maximum likelihood
estimator $\hatt p=x/n$, the BIC approximation yields
\beqn
{x!\,(n-x)!\over (n+1)!}
\doteq \Bigl({x\over n}\Bigr)^x \Bigl({n-x\over n}\Bigr)^{n-x}
   \Bigl\{ {x\over n}{n-x\over n}\Bigr)\Bigr\}^{1/2}
   {\ubiq \over \rootn}. 
\eeqn
Instructively, this is seen to break down
into Stirling for $x!$, Stirling for $(n-x)!$,
and Stirling for $(n+1)!$, in partes tre! 

\section{Notes on the historical background for different themes} 
\label{section:history}


In our introduction section we briefly pointed to \citet{Pearson24},
who ventured that Stirling did not quite deserve to
have his and only his name attached to the famous
$n!$ formula, his point being that \citet{Demoivre1730}
was essentially onto the same formula, but without
the $\ubiq$ constant.
Other scholars have disagreed with Pearson in this particular
regard; see the detailed comments in
\citet{Tweddie22}, 
\citet{Tweddle84} (the latter writing from the University of Stirling),
\citet[Section 3]{LeCam86},
\citet{Dutka91}, \citet{Gelinas17}.
``Stirling greatly surprised de Moivre by the introduction
of $\pi$ into the calculation of the ratio
of the coefficient of the middle term in $(1+x)^n$ to the
sum of all coefficients'',
i.e.~of ${n\choose n/2}(\half)^n$, writes \citet{Cajori23},
reviewing the \citet{Tweddie22} book.
When going through the historical details associated
with the $\log n!$ formula, from (\ref{eq:stirling}),
complete with an error term of the right order $O(1/n)$,
\citet{Gelinas17} argues, 
``Likewise indeed, James Stirling must keep his indisputable
priority for the discovery, proof, and publication of Stirling's
series [the formula for $\log n!$] including of course
the explicit constant $\log\ubiq$.''

In Section \ref{section:uniforms} we found new uses
for the Irwin--Hall distribution, so named due to two
papers published in the same Biometrika 1927 volume,
\citet{Irwin27} and \citet{Hall27}. 
E.S.~Pearson put in an informative editorial note inside
the latter article (pp.~243--244): 
``Proceeding by completely different methods, Irwin and Hall
have obained equations for the frequency distribution of means
in samples from a population following a `rectangular'
law of distribution [...]. Their results are of considerable
interest, but for the light they throw upon the distribution
of moments in samples from populations in which the
variable lies within a finite or limited range.'' 
As is clear from careful historical notes in
e.g.~\citet{Sheynin73} and \citet*{Marengo17}, however,
the distribution is in reality far older, going back
to Lagrange and Laplace in the latter 18th and the early 19th
century, associated also with early uses of generating functions. 

In Section~\ref{sec:stirlingagain} we derived Stirling{'}s formula
by combining the CLT for the median of uniforms and the
trapezoidal rule (see Eq.~\ref{eq:trapezoidal}). Intriguingly,
the trapezodial rule dates back at least to a century b.C.,
when Babylonian astronomers used it in tracking Jupiter,
see \citet{Ossendrijver16}.

In Section \ref{section:viamedians} we discussed
the Wallis 1656 product formula, and its relation
to both Stirling and the de Moivre binomial middle
probability approximation. 
Wallis' crucial role in developing infinitesimals,
including boldly inventing the symbols $\infty$ and ${1\over \infty}$,
in an intellectual era where thinking along such lines
was partly considered precarious and even heretic, 
is vividly described in \citet[Chs.~8, 9]{Alexander14} -- how
a {\it dangerous} mathematical theory shaped the modern world,
no less. There were several intellectual heavy-weights
among his antagonists, from Thomas Hobbes to
the politically powerful Jesuite school. 

In our Section \ref{section:Laplace}, with various
Stirling proofs based on Laplace approximations
to integrated likelihoods, we pointed to 
\citet{DiaconisFreedman86}, whose argument essentially
matches the first of our proofs in the section mentioned.
The \citet{DiaconisFreedman86} argument is at the outset
a mathematical one, not related to probability or
statistics per se. It is incidentally the Laplace
approximation itself which pushes the analysis
towards the $\ubiq$ constant, via the classical
Gaussian integral.  They explain however how they
``stumbled on'' their proof as a by-product of developing
theory for finite forms of de Finetti's theorem
for exponential families; see \citet{DiaconisFreedman80}.
In the spirit pointing to crucial achievements
in the distant past we ought also to point to
\citet{Laplace1774}, the origin of the integration
approximation technique carrying his name. 

The treatment of Stirling in \citet{Feller68} deserves
a footnote, since the book is a classic, carefully studied
in previous generations. He arrives at the Stirling approximation
in several steps, first examining the logarithm (page 52)
and then quite a bit later using the normal approximation
to the binomial (page 180) to infer that the constant involved
must be $\ubiq$; for some of the technicalities
he also credits \citet{Robbins55}. Interestingly,
he tosses out Exercise 22, page 66, concerning Laplace approximation
of the gamma function -- seemingly not noticing
that this is in fact the Stirling formula!, and, incidentally,
close to the main argument used in \citet{DiaconisFreedman86}. 

\section{Concluding remarks} 
\label{section:concludingremarks}

\def\cf{{\rm ch}}
\def\cf{\varphi}

In our article we have demonstrated that
the classic \citet{Stirling1730} approximation formula (\ref{eq:stirling}), 
the middle binomial probability formula (\ref{eq:middlebinomial})
associated with \citet{Demoivre1730} (though he did not know
the right constant), and the \citet{Wallis1656}
product formula (\ref{eq:wallis}), are essentially equivalent;
starting with any one of these historically important achievements,
one may deduce the others. 
We round off our article by offering a few concluding remarks.


\smallskip
{\it Remark A.}
Various classical and not too hard calculations
for teaching statistics concern the mean squared error
when estimating the binomial $p$, based on $X\sim\binom(n,p)$.
It is of course $\E\,(X/n-p)^2=p(1-p)/n$.
But in a Teacher's Corner contribution, \citet{Blyth80}
asks about the ``somewhat more natural measure''
$D_n(p)=\E_p\,|X/n-p|$, i.e.~using absolute loss. As we
know this is less easy mathematically, and it takes some
stamina, here Blyth actually uses results from
\citet{Frisch24}, to establish exact formulae for
$d_n$, the maximum risk over $p$. Correcting a little
mistake in his formulae, we find 
\beqn
d_n=\rootn\max_{{\rm all\ }p} D_n(p)
=\rootn {n-1\choose \half n-1}(\half)^n 
\eeqn
for even $n$. Blyth goes on, via the Stirling formula, to show that
$d_n\arr d=\half(2/\pi)^{1/2}=1/\ubiq$, explicitly using
the Stirling formula in the process. But we may also
utilise the basic reasoning of Section \ref{section:CLTcases}
to argue that we {\it know} that $\rootn\,\E\,|X_n/n-p|$
must tend to $(2/\pi)^{1/2}\{p(1-p)\}^{1/2}$,
via the CLT for binomials, and hence deduce from $d_n\arr d$
that, once more, Stirling holds. 

\smallskip
{\it Remark B.}
Consider a symmetric random walk, with steps to the right and the left
with equal probabilities, starting at zero. The position at
time $n$ is $Y_n=R_n-L_n=2R_n-n$, with $R_n$ a binomial $(n,\half)$.
The chance of revisiting zero after $2n$ steps is
\beqn
p_n=\Pr(Y_{2n}=0)={2n\choose n}(\half)^{2n}, 
\eeqn 
which we have met and worked with in Section \ref{sec:stirlingagain}.
de Moivre could prove the approximation $c/\rootn$, but
without finding the right $c$ constant. The point now is to
derive the Stirling related approximtaion $1/(\pi\,\rootn)$ again,
in yet another way. Let $Z_{2n}=2\,(2n)^{1/2}\{R_{2n}/(2n)-\half\}$,
which tends to the standard normal. We have
\beqn
p_n=\Pr(|Y_{2n}|<\half)=\Pr(|Z_{2n}|<1/(2n)^{1/2})
   \doteq 2\phi(0)\,1/(2n)^{1/2}=1/(\pi\,\rootn), 
\eeqn 
once more with `$\doteq$' in the precise sense that the
ratio between the two sides tends to 1. Rounding off this
remark, we mention that this particular approximation,
in the random walk context, can be used to approximate
the number of visits $\E\,N_{2n}$ to zero, in the course
of time steps $1,\ldots,2n$, as
$(1/\sqrt{\pi})\sumin 1/\sqrt{i}\doteq (1/\sqrt{\pi})\half\rootn$.
The point here is also that this number indeed tends to
infinity, related of course to the recurrency property
of the symmetric random walk; each state will be revisited
infinitely many times, with probability~1. 

\smallskip
{\it Remark C.} 
As in Section \ref{section:CLTcases}, consider 
$Y_n$ a Poisson with parameter $n$. Via results worked out
in \citet{Crow58}, one may deduce a perhaps surprisingly
simple expression for the mean absolute deviation,
$\E\,|Y_n-n|=2\exp(-n)n^{n+1}/n!$. Via the CLT and uniform
integrability, we must have $\E\,|Y_n-n|/\rootn\arr (2/\pi)^{1/2}$,
the absolute mean of a standard normal. 
We have rediscovered the Stirling, for the $(n+1)$st time
in this article.







\bibliographystyle{biometrika}
\bibliography{diverse_bibliography2024}

\end{document}